\documentclass[12pt]{amsart}

  \usepackage{amscd}
\usepackage{multirow}
\usepackage[table]{xcolor}
\def\NZQ{\Bbb}               

\def\RR{{\NZQ R}}

\def\frk{\frak}               

\def\Phi{{\frk n}}
\def\Phi{{\frk N}}
\def\opn#1#2{\def#1{\operatorname{#2}}} 
\opn\chara{char} \opn\length{\ell} \opn\pd{pd} \opn\rk{rk}
\opn\projdim{proj\,dim} \opn\injdim{inj\,dim} \opn\rank{rank}
\opn\depth{depth} \opn\grade{grade} \opn\height{height}
\opn\embdim{emb\,dim} \opn\codim{codim}

\opn\Tr{Tr} \opn\bigrank{big\,rank}
\opn\superheight{superheight}\opn\lcm{lcm}
\opn\trdeg{tr\,deg}
\opn\reg{reg} \opn\lreg{lreg} \opn\ini{in} \opn\lpd{lpd}
\opn\size{size} \opn\des{des}
\opn\div{div} \opn\Div{Div} \opn\cl{cl} \opn\Cl{Cl}
\opn\Spec{Spec} \opn\Supp{Supp} \opn\supp{supp} \opn\Sing{Sing}
\opn\Ass{Ass} \opn\Min{Min}
\opn\Ann{Ann} \opn\Rad{Rad} \opn\Soc{Soc}
\opn\Im{Im} \opn\Ker{Ker} \opn\Coker{Coker} \opn\Am{Am}
\opn\Hom{Hom} \opn\Tor{Tor} \opn\Ext{Ext} \opn\End{End}
\opn\Aut{Aut} \opn\id{id}

\opn\nat{nat}
\opn\pff{pf}
\opn\Pf{Pf} \opn\GL{GL} \opn\SL{SL} \opn\mod{mod} \opn\ord{ord}
\opn\Gin{Gin} \opn\Hilb{Hilb} \opn\ex{ex}
\opn\aff{aff} \opn\con{conv} \opn\relint{relint} \opn\st{st}
\opn\lk{lk} \opn\cn{cn} \opn\core{core} \opn\vol{vol}
\opn\link{link} \opn\star{star}
\opn\gr{gr}

\def\pot#1#2{#1[\kern-0.28ex[#2]\kern-0.28ex]}

\opn\dirlim{\underrightarrow{\lim}}
\opn\inivlim{\underleftarrow{\lim}}
\def\Implies{\ifmmode\Longrightarrow \else
        \unskip${}\Longrightarrow{}$\ignorespaces\fi}
\def\implies{\ifmmode\Rightarrow \else
        \unskip${}\Rightarrow{}$\ignorespaces\fi}
\def\iff{\ifmmode\Longleftrightarrow \else
        \unskip${}\Longleftrightarrow{}$\ignorespaces\fi}

\let\:=\colon
\newtheorem{Theorem}{Theorem}[section]
\newtheorem{Lemma}[Theorem]{Lemma}
\newtheorem{Corollary}[Theorem]{Corollary}
\newtheorem{Proposition}[Theorem]{Proposition}

\newtheorem{Example}[Theorem]{Example}

\newtheorem{Definition}[Theorem]{Definition}
\newtheorem{Problem}[Theorem]{Problem}

\opn\Syz{Syz} \opn\Im{Im} \opn\Ker{Ker} \opn\Coker{Coker}
\opn\Am{Am} \opn\Hom{Hom} \opn\Tor{Tor} \opn\Ext{Ext} \opn\End{End}
\opn\Aut{Aut} \opn\id{id}

\opn\nat{nat}
\opn\pff{pf}
\opn\Pf{Pf} \opn\GL{GL} \opn\SL{SL} \opn\mod{mod} \opn\ord{ord}
\opn\Gin{Gin}\opn\min{min}
\opn\Hilb{Hilb}\opn\adeg{adeg}\opn\std{std}\opn\ip{infpt}
\opn\Pol{Pol}\opn\sdepth{sdepth}\opn\infpt{infpt}
\opn\depth{depth}\opn\sqdepth{sqdepth}\opn{\Mon}{Mon}
\opn\sd{sd}\opn\D{D}
\let\epsilon\varepsilon
\let\phi=\varphi
\let\kappa=\varkappa
\opn\dis{dis}
\def\pnt{{\raise0.5mm\hbox{\large\bf.}}}

\opn\Lex{Lex}

\title{On Partial Barycentric Subdivision}

\thanks{$^{1}$Partially supported by Higher Education Commission of Pakistan under the program HEC post doctoral fellowship phase II(batch IV), ref 2-4(45)/PDFP/HEC/2010/2}

\author{Sarfraz Ahmad$^1$}
\address{COMSATS Institute of Information Technology, Lahore, Pakistan}
\email{sarfrazahmad@ciitlahore.edu.pk}

\author{Volkmar Welker}
\address{Fachbereich Mathematik und Information, Philipps-Universit\"{a}t Marburg, 35032 Marburg, Germany}
\email{welker@mathematik.uni-marburg.de}

\begin{document}

    \begin{abstract}
      The $l$\textsuperscript{th} partial barycentric subdivision is defined for a $(d-1)$-dim\-ensional
      simplicial complex $\Delta$ and studied along with its combinatorial, geometric
      and algebraic aspects. We analyze the behavior of the $f$- and $h$-vector under the
      $l$\textsuperscript{th} partial barycentric subdivision
      extending previous work of Brenti and Welker on the standard barycentric subdivision -- the case $l = 1$. We discuss
      and provide properties of the transformation matrices sending the $f$- and $h$-vector of $\Delta$ to the
      $f$- and $h$-vector of its
      $l$\textsuperscript{th} partial barycentric subdivision. We conclude
      with open problems.
    \end{abstract}


\maketitle
\section{Introduction}

    \noindent For a $(d-1)$-dimensional simplicial complex $\Delta$ on the ground set $V$ the barycentric subdivision
    $\sd(\Delta)$ of $\Delta$ is the simplicial complex on the ground set $V\setminus \{\emptyset\}$ with simplices the flags
    $A_0\subset A_1 \subset\cdots\subset A_i$ of elements $A_j\in \Delta\setminus\{\emptyset\},$ $0\leq j \leq i$.
    For $1\leq l\leq d$, we define the $l$\textsuperscript{th} partial barycentric subdivision of $\Delta$. This is
    a geometric subdivision, in the sense of \cite{St}, such that
    $\sd^{l-1}(\Delta)$ is a refinement of $\sd^l(\Delta)$, $\sd^d(\Delta) = \Delta$ and $\sd^1(\Delta)=\sd(\Delta)$. Roughly
    speaking, the $l$\textsuperscript{th} partial barycentric subdivision arises
    when only the simplices of dimension $\geq l$ are barycentrically subdivided.
    In the paper, we provide a detailed analysis of the effect of the $l$\textsuperscript{th} barycentric subdivision operation on
    the $f$- and $h$-vector of a simplicial complex.
    Most enumerative results will be related to refinements of permutation statistics for the symmetric group.
    Our results extend the results from \cite{BW} for the case $l =1$. We refer the reader also to \cite{D} and \cite{N} for more
    detailed information in that case.

    The paper is organized as follows.
    We start in Section 2 with geometric and combinatorial descriptions of the $l$\textsuperscript{th} partial barycentric
    subdivision and its implications on the generators of the Stanley-Reisner ideal of the complex.
    In Section 3 we study the enumerative combinatorics of the $l$\textsuperscript{th} partial barycentric subdivision.
    In particular, we relate
    in Lemma \ref{le:fvector} and Theorem \ref{thm:main} the effect of the $l$\textsuperscript{th} barycentric
    subdivision on the $f$- and $h$-vector of the simplicial complex $\Delta$ to a permutation statistics
    refining the descent statistics.
    In Section 4 we analyze the transformation matrices sending the $f$- and $h$-vector of the simplicial complex
    $\Delta$ to the corresponding vector for the $l$\textsuperscript{th} barycentric subdivision. We show that both maps
    are diagonizable and provide the eigenvalue structure.
    Note that by general facts the two matrices are similar.
    The main result of this section, Theorem \ref{th:eigenvector}, shows that the eigenvector corresponding to the highest
    eigenvalue of the $h$-vector transformation
    can be chosen such that it is of the form $(0,b_1,\ldots,b_{d-1},0)$ for strictly positive numbers $b_i$, $1\leq i\leq d-1$.
    In Section 5 we present some open problems. We ask for explicit descriptions of the eigenvectors and then shift the focus to the
    local $h$-vector which has been introduced by Stanley
    \cite{St}. The local $h$-vector is a measure for the local effect of a subdivision operation. In particular, general results by Stanley,
    predict that
    the local $h$-vector for the $l$\textsuperscript{th} partial barycentric subdivision is non-negative. For $l =1$ the
    local $h$-vector was computed by Stanley in terms of the
    excedance statistics on derangements. We exhibit some computations and possible approaches to the local $h$-vector for the
    $l$\textsuperscript{th} barycentric subdivision in general.

\section{The $l$\textsuperscript{th} partial barycentric subdivision}

  \subsection{Geometric definition}
    \noindent We first give a geometric definition of the $l$\textsuperscript{th} partial barycentric subdivision.
    For that we recall some basic facts about the reflection arrangement of the symmetric group $S_d$
    permuting the $d$ letters from $[d] := \{ 1,2,\ldots,d\}$. The reflection arrangement
    $\mathcal{B}_d$ in $\RR^d$ of the symmetric group $S_d$ consists of the hyperplanes
    $H_{ij} = \{ (x_1,\ldots, x_d) \in \RR^d~:~x_i-x_j=0 \}$, $1\leq i<j\leq d$.
    To each permutation $w \in S_d$ there
    corresponds a region $R_w$ of $\mathcal{B}_d$ given by
    $$R_w=\{(\lambda_1,\ldots,\lambda_d)\in \mathbb{R}^d:\lambda_{w(1)} > \lambda_{w(2)} >  \cdots > \lambda_{w(d)}\}.$$
    Hence the number of regions of $\mathcal{B}_d$ is $d!$.
    We write $R_{w,+}$ for the intersection of $R_w$ with $\RR_{\geq 0}^d$. It is easily seen
    that geometrically the closure of $R_{w,+}$ is a simplicial cone.

    The intersection of the closures of the cones $R_{w,+}$, $w \in S_d$, and the standard $(d-1)$-simplex $\Delta_{d-1} =
    \{ (\lambda_1, \ldots, \lambda_d) \in \RR^d ~|~\lambda_1+\cdots \lambda_d = 1, x_i \geq 0, 1 \leq i \leq d\}$ induces
    a simplicial decomposition of $\Delta_{d-1}$. This decomposition is called the \textit{barycentric
    subdivision} of $\Delta_{d-1}$ and is denoted by $\sd(\Delta_{d-1})$.
    We are interested in a sequence $\sd^l(\Delta_{d-1})$, $1 \leq i \leq d$, of simplicial subdivisions of the simplex,
    which have the property that $\sd^1(\Delta_{d-1}) = \sd(\Delta_{d-1})$ and  $\sd^{l-1}(\Delta_{d-1})$ is a
    refinement of $\sd^l(\Delta_{d-1})$.

    \noindent For $1 \leq l \leq d$, we set $S_d^l$ to be the set of permutations $w \in S_d$ for which
    $w(1) > \cdots > w(l)$. We define the $l$-cone $R_w^l$ of a $w \in S_d^l$ to be
    $$R_w^l=\{(\lambda_1,\ldots,\lambda_d)\in \RR^d~:~\lambda_{w(1)},\ldots,\lambda_{w(l)} > \lambda_{w(l+1)} > \cdots  >
       \lambda_{w(d)}\}.$$
    Clearly $R_w^l$ is a cone.
    We write $R_{w,+}^l$ for the intersection of $R_w^l$ with
    $\RR_{\geq 0}^d$. Again the closure of $R_{w,+}^l$ is a simplicial cone which is the
    union of all closures of the $R_{v,+}$ for $v \in S_d$ such that $v(i) = w(i)$ for
    $l+1 \leq i \leq n$.
    We call the simplicial decomposition induced by the collection of all $R_{w,+}^l$ for
    $w \in S_{d}^l$ on $\Delta_{d+1}$ the \textit{$l$\textsuperscript{th} partial barycentric subdivision} of $\Delta_{d-1}$
    and denote it by $\sd^l (\Delta_{d-1})$.
    Obviously, we have that $\sd^d(\Delta_{d-1}) = \Delta_{d-1}$, $\sd^{1}(\Delta_{d-1}) = \sd(\Delta_{d-1})$
    and $\sd^{l-1}(\Delta_{d-1})$ is a refinement of $\sd^l(\Delta_{d-1})$.
    If $l > d$ then we set $\sd^l (\Delta_{d-1}) = \Delta_{d-1}$.
    For a $(d-1)-$ dimensional simplicial complex $\Delta$ on the vertex set
    $V=[n]$ its $l$\textsuperscript{th} partial barycentric subdivision is the complex $\sd^l(\Delta)$
    which is the subdivision of $\Delta$ obtained by replacing each simplex by its $l$\textsuperscript{th} partial
    subdivision. Roughly speaking this means that we cone all $(k-1)$-faces of $\Delta$ over their barycenters for
    all $l \leq k$.

    By construction the number of cones $R_w^l$, $w \in S_d^l$, is
    $\frac{d!}{(d-l)!} = d \cdot (d-1)\cdots (d-l+1)$.
    Next, we want to get a better understanding of the facial structure of $\sd^l(\Delta_{d-1})$.

    We have already seen that the $(d-1)$-dimensional faces are in bijection with the permutations in $S_d^l$.
    We turn this description into a description by combinatorial objects that are more suitable for studying all
    faces of $\sd^l(\Delta_{d-1})$.
    We identify a permutation
    $w \in S_d^l$ with a formal chain
    $$w(1) , w(2) ,\ldots,w(l) > w(l+1) > \cdots > w_{d}$$ 
    and this chain in turn with $A_0 \subset A_1 \subset \cdots \subset A_{d-l}$
    for $A_i = \{w(1),\ldots, w(l), \ldots, w(l+i)\}$.

    Using this chain description, a $2$-dim face of $\sd^l(\Delta_2)$ corresponding to
    $w = 1~2~3$ is either
    $\{1,2,3\}$ for $l = 3$ or $\{1,2\} \subset \{1,2,3 \}$ for $l = 2$ or $\{1\} \subset
    \{1,2\} \subset \{1,2,3\}$ for $l = 1$.

    More generally, the $(i-1)$-faces of $\sd^l(\Delta)$ are indexed by chains
    $A_0 \subset A_1 \subset \cdots \subset A_r$
    for which:

    \begin{tabular}{cc}
       \begin{minipage}{0.5\textwidth}
         \begin{itemize}
           \item[(C1)]  $0 \leq \# A_0 \leq l$,
           \item[(C3)]  $l+1 \leq \# A_1$.
         \end{itemize}
       \end{minipage}
      &
       \begin{minipage}{0.5\textwidth}
         \begin{itemize}
           \item[(C2)]  $\#A_0 + r = i$,
           \item[(C4)]  $A_r \in \Delta$.
         \end{itemize}
       \end{minipage}
    \end{tabular}

    At the beginning of Section \ref{se:transformation} we will further reformulate this description in terms of yet another
    combinatorial objects.

    Geometrically, the face of $\sd^l(\Delta_{d-1})$ corresponding to
    $A_0 \subset A_1 \subset \cdots \subset A_0$ is the set of points
    $(\lambda_1,\ldots,\lambda_{d}) \in \Delta_{d-1}$ for which

    \begin{itemize}
      \item[(i)] We have $\lambda_i = \lambda_j$ if $i,j \in A_s \setminus A_{s-1}$ for some $1 \leq s \leq r$.
      \item[(ii)] We have $\lambda_i > \lambda_j$ if $i \in A_s$ and $j \in A_t$ for some $0 \leq s < t \leq r$.
      \item[(iii)] We have $\lambda_i = 0$ if $i \not\in A_r$.
    \end{itemize}

    In particular, a vertex $v$ of the $l$\textsuperscript{th} partial barycentric subdivision $\sd^{l}(\Delta)$ of a $(d-1)$-dimensional simplicial complex
    $\Delta$ either belongs to the vertex set of $\Delta$ or can be identified with an
    $(m-1)$-face $v= \{ v_{i_1}, v_{i_2}, \ldots , v_{i_m}\} \in \Delta$ of $\Delta$ for some $l\leq m \leq d$.

    \subsection{Algebraic aspects}
    \noindent Let $\Delta$ be a $(d-1)$-dimensional simplicial complex on vertex set $[n]$. Let $k$ be a field and $R=k[x_1,\ldots,x_n]$ be the polynomial ring in $n$ variables.
    The \textit{Stanley-Reisner ideal} $I_{\Delta}$ is the ideal of $R$ generated by
    the squarefree monomials $\prod_{i \in \mathcal{N}} x_i$ whose index set $\mathcal{N}$ is a non-face of $\Delta$.
    A non-face $\mathcal{N}\not\in \Delta$ is called a \textit{minimal non-face} of $\Delta$ if no proper subset of $\mathcal{N}$ is a non-face of $\Delta$.
    It is easily seen that the generators in the unique minimal monomial generating set of the Stanley-Reisner ideal correspond to the minimal non-faces of $\Delta$.
    The quotient $R/I_{\Delta}$ is called the \textit{Stanley-Reisner ring} or \textit{face ring} and is denoted by $k[\Delta]$.

    Let  $\sd^l(\Delta)$ be the $l$\textsuperscript{th} partial barycentric subdivision of $\Delta$, where $0\leq l \leq d-1$. Similarly, let $I_{\sd^l(\Delta)}$ and $k[\sd^l(\Delta)]$ be the Stanley-Reisner ideal and
    face ring of $\sd^l(\Delta)$, respectively.
    Let $V_l=[n]\cup B_l$ be the set of vertices of $\sd^l(\Delta)$, where $B_l=\{b_1,\ldots,b_l\}$ is the set of barycenters of the faces of $\Delta$ whose dimension ranges
between $l$ and $d-1$.

    \begin{Lemma}
      \label{le:nonface}
      Let $\Delta$ be a $(d-1)$-dimensional simplicial complex on ground set $[n]$. For $1 \leq l \leq d$ let $\mathcal{N}$ be a subset of the vertex set of $\sd^l(\Delta)$. In case $\mathcal{N}$ is
      a minimal non-face of $\sd^l(\Delta)$ then either
      $\mathcal{N}\subseteq [n]$ and $\mathcal{N}$ is a minimal non-face of $\Delta$ or $|\mathcal{N}| = 2$.
    \end{Lemma}
    \begin{proof}
      If $\mathcal{N}\subset [n]$ and a minimal non-face of $\Delta$, then we are done. Since two vertices of $\Delta$ are connected by an edge in $\sd^l(\Delta)$ if and only if they
      are connected by an edge in $\Delta$ it follows that any subset of $[n]$ that is a minimal non-face of $\sd^l(\Delta)$ must be a minimal non-face of $\Delta$.
      Now, suppose there exists at least one vertex $b \in \mathcal{N} \setminus [n]$ and $|\mathcal{N}|\neq 2$. If $|\mathcal{N}| =1 $ then $\mathcal{N} = \{ b\}$ but $b$ is a vertex and
      hence a face. This leads to a contradiction and we are left with the case $|\mathcal{N}| \geq 2$. Let $\mathcal{F} \in \sd^l(\Delta)$ be a facet such that $b \in \mathcal{F}$. Then there exists at least one vertex
      $v\in \mathcal{N} \backslash \mathcal{F}$, otherwise $\mathcal{N}$ is no more a non-face. But by construction the edge $\{ b, v\}$ is not an edge in $\sd^l(\Delta)$. Thus $\{b_j,v\}\subset \mathcal{N}$
      from which $\{b_j,v\} = \mathcal{N}$ follows.
    \end{proof}

    Now we have the following Proposition:

    \begin{Proposition}
      Let $\Delta$ be a $(d-1)$-dimensional simplicial complex on ground set $[n]$. Then for $1 \leq l \leq d-1$ the Stanley-Reisner ideal $I_{\sd^l(\Delta)}$
      is generated by square free monomial ideals of degree at most $l+1$.
    \end{Proposition}
    \begin{proof}
      Suppose there exists some generator $u\in I_{\sd^l(\Delta)}$ with degree
      strictly larger than $l+1$.
      Then there exists a minimal non-face $\mathcal{N} \subset  \sd^l(\Delta)$
      such that $|\mathcal{N}|>l+1$. Thus by Lemma \ref{le:nonface} it follows
      that $\mathcal{N} \subset [n]$.
      But in $\sd^l(\Delta)$, we have coned all faces
      $\mathcal{F} \subseteq [n]$ of dimension $\geq l$ over their barycenters. Thus there does not
      exist any non-face of dimension greater
      than $l$. Therefore, $|\mathcal{N}|\leq l+1,$ a contradiction.
    \end{proof}

  \section{$f$-vector and $h$-vector transformation}
    \label{se:transformation}
    In this section we study the transformation maps sending the $f$- and $h$-vector
    of a simplicial complex $\Delta$ to the $f$- and $h$-vector of the
    $l$\textsuperscript{th} partial barycentric subdivision of $\Delta$.

    We consider set systems $B=|B_0|B_1|...|B_{r}|$ such that
    $B_1, \ldots, B_r \neq \emptyset$ and $B_s \cap B_t = \emptyset$ for $0 \leq s < t
    \leq r$. Despite the fact that $B_0$ can be empty we call such a system an
    ordered set partition. We write $R(j,i,l)$ for the number of such ordered
    set partitions $B=|B_0|B_1|...|B_{r}|$ for which
    \begin{enumerate}
      \item[(P1)] $B_0 \cup \cdots \cup B_r = [j]$
      \item[(P2)] If $r \geq 1$ then $\# (B_0 \cup B_1) \geq l+1$
      \item[(P3)] $\# B_0 + r = i$.
    \end{enumerate}
    and call $R(j,i,l)$ the \textit{restricted Stirling number} for
    the parameters $j$,$i$,$l$.
    For $l = 0$ by $\# (B_0 \cup B_1) \geq l+1 = 1$ it follows from
    $B_1 \neq \emptyset$ that $B_0 = \emptyset$.
    Hence $r = i$ and $|B_1|\cdots |B_i|$ is a usual ordered set partition
    of the $j$-element set $B_1 \cup \cdots \cup B_i$ into
    $i$ (non-empty) blocks. Thus $R(j,i,0)=i!S(j,i)$, where $S(i,j+1)$
    is the Stirling number of second kind.

    Recall that the $f$-vector $f^\Delta = (f_{-1}^\Delta, \ldots, f^\Delta_{d-1})$
    of a $(d-1)$-dimensional
    simplicial complex is the vector with its $i$\textsuperscript{th} entry
    $f_i^\Delta$ counting the $i$-dimensional faces of $\Delta$.
    Using this notation, we have the following lemma.

    \begin{Lemma}
    \label{le:fvector}
      Let $\Delta$ be a $(d-1)-$dimensional simplicial complex with
      $f$-vector $f^\Delta=(f^\Delta_{-1},\ldots,f^\Delta_{d-1})$.
      Then $$f_{i-1}^{\sd^l(\Delta)}=\sum\limits_{j=0}^d f^\Delta_{j-1} \cdot R(j,i,l).$$
    \end{Lemma}
    \begin{proof}
      Let $A_0 \subset A_1 \subset \cdots \subset A_r$ be an $(i-1)$-dimensional
      face of $\sd^l(\Delta)$.
      Then by (C2) we have $i = \# A_0 + r$. We set $j = \# A_r$ and assume
      without loss of generality that $A_r = [j]$.
      We define $B_0 = A_0$ and $B_s = A_{s} \setminus A_{s-1}$ for $1 \leq s \leq r$.
      Then $i = \# B_0 + r$ and $j = \# (B_0 \cup \cdots \cup B_r)$.
      If $r \geq 1$ then by (C3) $\# A_1 = \# (B_0 \cup B_1) \geq l+1$.
      Hence $|B_0| \cdots |B_r|$ satisfies (P1)-(P3). Conversely,
      if $|B_0| \cdots |B_r|$ satisfies (P2)-(P3) and
      $(B_0 \cup \cdots \cup B_r) \in \Delta$ then one easily checks that
      for $A_s = B_0 \cup \cdots \cup B_s$, $0 \leq s \leq r$,
      the chain $A_0 \subset \cdots \subset A_r$ satisfies (C1)-(C4).

      Thus for any $(j-1)$-dimensional face $F$ of $\Delta$ on gets
      $R(j,i,l)$ faces $A_0 \subset \cdots \subset A_r$ of dimension $i-1$
      in $\sd^l(\Delta)$ with $F = A_r$.
    \end{proof}

    \noindent Next we study the transformation of the $h$-vector. Recall that the
    $h$-vector of a $(d-1)$-dimensional simplicial complex $\Delta$ is the
    integer vector $h^\Delta = (h^\Delta_0,\ldots, h^\Delta_d)$ for
    which $\sum_{i=0}^d f^\Delta_{i-1}x^{d-i}
    = \sum_{i=0}^d h^\Delta_i x^{d-i}$, where $x$ is some indeterminate.
    For a permutation $\sigma \in S_d$ we denoted by
    $$\D(\sigma)=\{i\in [d-1]~ |~ \sigma(i)>\sigma (i+1)\}$$
    its decent set and write $\des(\sigma):= \# \D(\sigma)$ for its number of descents.
    Following \cite{BW} for $d \geq 1$ and integers $i$ and $j$ we denote by
    $A(d,i,j)$ the number of permutations
    $\sigma \in S_d$ such that $\sigma(1)=j$ and $\des (\sigma)=i$.
    In particular, $A(d,i,j)=0$ if $i\leq -1$ or $i\geq d$.

    In the sequel, we define a refinement of the preceding statistics suitable
    for the study of our $h$-vector transformation.
    By definition of $S_d^l$ we have the following strictly increasing chain
    of subgroups:
    $$S_{d+1}^{d-1}\subset S_{d+1}^{d-2}\subset \cdots \subset S_{d+1}^{2}
      \subset S_{d+1}^{1}=S_{d+1}.$$

    We define the \emph{$l$-descent set} $D^l(\sigma)$ of a permutation
    $\sigma \in S_d^l$ as follows:
    \begin{Definition}
      A number $i\in [d-1]$ belongs to the $l$-descent set $D^l(\sigma)$ of
      $\sigma\in S_d^l$, if $i$ satisfies one of the following two conditions.
      \begin{enumerate}
        \item $i\in [l]$ and $\sigma(i)> \sigma(l+1)$,
        \item $i\in [d-1]\setminus [l]$ and $\sigma(i)> \sigma(i+1).$
      \end{enumerate}
    \end{Definition}
    We write $\des^l(\sigma)=|D^l(\sigma)|$ for the \emph{number of $l$-descents}
    of a permutation $\sigma \in S_d^l$.
    Note that for $l = 1$ condition (1) is equivalent to having a decent in position
    one and therefore $D^1(\sigma)$ is just the usual descent set of $\sigma$.

    \begin{Example}
      Let $\sigma_1,\sigma_2,\sigma_3 \in S_6^4$, such that,
      $\sigma_1 =(\underline{4,3,2,1},6,5)$,
      $\sigma_2 =(\underline{6,5,2,1},3,4)$ and
      $\sigma_3 =(\underline{6,4,3,2},5,1)$. Then

      \begin{center}
       $
       \begin{array}{ccc}
         D^4(\sigma_1)=\{5\} &
         D^4(\sigma_2)=\{1,2\} &
         D^4(\sigma_3)=\{1,5\} \\
         \des^4(\sigma_1)=1 &
         \des^4(\sigma_2)=2 &
         \des^4(\sigma_3)=2.
       \end{array}
       $
      \end{center}
    \end{Example}

    For all $d\geq 1,$ $1\leq l \leq d-1$ and all integers $i$ and $j$ we
    denote by $A(d,i,j,l)$ the number of permutations $\sigma \in S_d^l$ such that
    $\des^l(\sigma)=i$ and $\sigma(d+1)=j$. Note that $A(d,i,j,l)=0$ if
    $i\leq -1$ or $i\geq d$.

    The following is our first main result.
    The case $l=1$ was treated in \cite[Thm. 1]{BW}.

    \begin{Theorem}
      \label{thm:main}
      Let $\Delta$ be a $(d-1)$-dimensional simplicial complex. Then
      $$h_j^{\sd^l(\Delta)}=\sum_{\mu=0}^d A(d+1,j,d+\mu-1,l)h_{\mu}^{\Delta}$$
      for $1\leq l \leq d-1$ and $0\leq j \leq d.$
    \end{Theorem}
    \begin{proof}
      For all $0\leq j\leq d$, we have:
      \begin{eqnarray}
      \nonumber   h_j^{\sd^l(\Delta)} &=& \sum_{i=0}^{j}{d-i \choose j-i}(-1)^{j-i}f_{i-1}^{\sd^l(\Delta)} \\
      \nonumber  &=& \sum_{i=0}^j {d-i \choose j-i} (-1)^{j-i} \sum_{k=0}^d f_{k-1}^{\Delta}R(k,i,l) \\
      \nonumber  &=& \sum_{i=0}^j\sum_{k=0}^d {d-i \choose j-i} (-1)^{j-i}R(k,i,l)\sum_{\mu=0}^k {d-\mu \choose d-k}h_{\mu}^{\Delta}\\
      \nonumber  &=& \sum_{\mu=0}^d{(}\sum_{k=0}^d\sum_{i=0}^j (-1)^{j-i}{d-i \choose j-i}{d-\mu \choose d-k}R(k,i,l){)}h_{\mu}^{\Delta}.
      \end{eqnarray}
      Fix a permutation $\sigma \in S_d^l$ and let
      $\D^l(\sigma) =\{s_1,\ldots,s_k\}\subseteq [d]$
      with $s_1<\cdots < s_k$ its descent set.
      We define $p_{\sigma}$ be $0$ if there is no $s_{p_{\sigma}}\in [l]$
      for which $\sigma(s_{p_{\sigma}}) > \sigma(l+1)$ and to be the maximal
      number $p_\sigma \in [k]$ for which $s_{p_{\sigma}}\in [l]$
      and $\sigma(s_{p_{\sigma}}) > \sigma(l+1)$.
      Thus by definition of an $l$-descent any set $S$ that can arise as the descent
      set of a permutation $\sigma \in S_d^l$ is of the form
      $S = \{1,\ldots, p, s_{p+1}, \ldots, s_k\}$ for
      some $0 \leq p_\sigma < l \leq s_{p}+1$.
      We want to count
      $$\left\{\sigma \in S_{d+1}^l\,\,\,\,\ \Big| \,\,\,\,\,\,\
      \begin{array}{ll}
         \D^l(\sigma)\subseteq S,p_{\sigma}=p  \\
         \sigma(d+1)=d+1-\mu
       \end{array}
      \right\}.
      $$
      First, we count the possibilities to choose
      $\sigma(s_{p+1}+1) , \ldots, \sigma(d+1)$. In this range a descent is just a
      usual descent and we have ${s_k \choose s_{p+1},s_{p+2}-s_{p+1},
      \ldots,s_k-s_{k-1}}{d-\mu \choose d-t_k}$ possibilities.
      Having done this we are left with $s_{p+1}$ elements that we have to
      arrange accordingly. Let $U$ be this set of elements.
      Let us consider the options for $\sigma(l+1)$. In order to create no
      descent in the range of $l+1$ and $s_{p+1}-1$ there have
      to be at least $s_{p+1} - (l+1)$ elements larger than $\sigma(l+1)$.
      Hence there can be at most $l$ elements smaller than $\sigma(l+1)$.
      In order to create no descent between $\sigma(i)$ and $\sigma(l+1)$ for
      some $p < i < l+1$ there must be at least $l-p$ elements
      smaller than $\sigma(l+1)$. Hence $\sigma(l+1)$ can only range from
      the $(l+1-p)$\textsuperscript{th} element of $U$ to the
      $(l+1)$\textsuperscript{st} element of $U$.
      Let $l+1-p \leq i \leq l+1$ and assume $\sigma(l+1)$ is the
      $i$\textsuperscript{th}
      element of $U$. Then in order to fix the permutation $\sigma$ we have
      to fix $s_{p+1} -(l+1)$ element larger than $\sigma(l+1)$ that will become the
      images of $\sigma(l+2),\ldots, \sigma(s_{p+1})$ in
      increasing order. For that we have ${s_{p+1} -i \choose s_{p+1} - (l+1)}$
      possibilities.
      Using $l+1 \leq s_{p+1}$ we obtain:

      $$\# \left\{\sigma \in S_{d+1}^l\,\,\,\,\ \Big| \,\,\,\,\,\,\
      \begin{array}{ll}
         \D^l(\sigma)\subseteq S,p_{\sigma}=p  \\
         \sigma(d+1)=d+1-\mu
       \end{array}
      \right\}
      $$
      $$=\sum_{i=l+1-p}^{l+1}{s_{p+1}-i \choose s_{p+1}-(l+1)}
                 {s_k \choose s_{p+1},s_{p+2}-s_{p+1},\ldots,s_k-s_{k-1}}
                 {d-\mu \choose d-t_k}$$
      but,
      $$\nonumber \sum_{i=l+1-p}^{l+1}{s_{p+1}-i \choose s_{p+1}-(l+1)}$$
      \begin{eqnarray}
      \nonumber &=& {s_{p+1}-l-1+p \choose s_{p+1}-l-1}+
         {s_{p+1}-l-2+p \choose s_{p+1}-l-1}+\cdots+
         {s_{p+1}-l-1 \choose s_{p+1}-l-1}\\
      \nonumber &=& {s_{p+1}-l-1+p+1 \choose s_{p+1}-l-1+1}=
                    {s_{p+1}-l+p \choose s_{p+1}-l}
      \end{eqnarray}
      so,
      $$\# \left\{\sigma \in S_{d+1}^l\,\,\,\,\ \Big| \,\,\,\,\,\,\
      \begin{array}{ll}
         \D^l(\sigma)\subseteq S,p_{\sigma}=p  \\
         \sigma(d+1)=d+1-\mu
       \end{array}
      \right\}
      $$
      $$={s_{p+1}-l+p \choose s_{p+1}-l}
         {s_k \choose s_{p+1},s_{p+2}-s_{p+1},\ldots,s_k-s_{k-1}}
         {d-\mu \choose s-s_k}$$
      therefore,
      $$\sum_{\{1\leq s_1<\cdots < s_k\leq d\}}\sum_{p=0}^l
          \# \left\{\sigma \in S_{d+1}^l\,\,\,\,\ \Big| \,\,\,\,\,\,\
          \begin{array}{ll}
            \D^l(\sigma)\subseteq S,p_{\sigma}=p \\
            \sigma(d+1)=d+1-\mu
          \end{array}
        \right\}
      $$
      $$=\sum_{1\leq s_1< \cdots <s_k\leq d}\sum_{p=0}^l
                 {{s_{p+1}-l+p \choose s_{p+1}-l}
                  {s_k \choose s_{p+1},s_{p+2}-s_{p+1},\ldots,s_k-s_{k-1}}
                  {d-\mu \choose d-s_k}}$$
      $$=\sum_{j=k}^d{d-\mu \choose d-j}\sum_{1\leq s_1< \cdots <s_{k-1}\leq j-1}
            \sum_{p=0}^l {s_{p+1}-l+p \choose s_{p+1}-l}
                         {j \choose s_{p+1},s_{p+2}-s_{p+1},\ldots,j-s_{k-1}}
      $$

      For a fixed sequence $s_1 < \cdots < s_k = j$ for which
      $s_1=1, \ldots, s_p = p$ and $s_{p+1} \geq l+1$ we set
      $B_0=[s_p]$ for $p \geq 1$ and $B_0=\emptyset$ for $p=0$. Also,
      we set
      $$B_{\omega}=[s_{p+\omega-1}+1,s_{p+\omega}]\text{ for }
                        1\leq \omega \leq k-p.$$
      Then $B_0 \cup \cdots \cup B_r = [j]$ which implies (P1).
      If $r \geq 1$ then $\# (B_0 \cup B_1)=\# B_0 + \# B_1 \geq p+ (l+1-p)=l+1$
      implying (P2). Finally $\# B_0 + r =p+(k-p)= k$ shows (P3).

      For notational convenience we set $s_0=0$.
      So for $r=k-p$,
      \begin{gather}
         \label{eq:part}
         \sum_{1\leq s_1< \cdots <s_{k-1}\leq j-1}\sum_{p=0}^l
                {s_{p+1}-l+p \choose s_{p+1}-l}
                {j \choose s_{p+1},s_{p+2}-s_{p+1},\ldots,j-s_{k-1}}
      \end{gather}
      counts the number of $B=|B_0|B_1|...|B_{r}|$ of $[j]$ satisfying (P1) - (P3).
      Therefore \eqref{eq:part} equals $R(j,k,l)$.

      Using this result we have,
      $$
        \sum_{S\subseteq [d],\# S =k}
           \# \left\{\sigma \in S_{d+1}^l\,\,\,\,\ \Big| \,\,\,\,\,\,\
        \begin{array}{ll}
          \D^l(\sigma)\subseteq S,  \\
          \sigma(d+1)=d+1-\mu
        \end{array}
        \right\}
      $$
      $$=\sum_{j=k}^d {d-\mu \choose d-j} R(j,k,l)$$
      therefore,
      $$\sum_{k=0}^d\sum_{i=0}^j (-1)^{j-i}
            {d-i \choose j-i}
            {d-\mu \choose d-k} R(k,i,l) $$
      \begin{eqnarray}
      \nonumber &=& \sum_{i=0}^j
           (-1)^{j-i}{d-i \choose j-i}\sum_{\{S\subseteq [d],\# S=i\}}
              \# \left\{\sigma \in S_{d+1}^l\,\,\,\,\ \Big| \,\,\,\,\,\,\
      \begin{array}{ll}
         \D^l(\sigma)\subseteq S,  \\
         \sigma(d+1)=d+1-\mu
      \end{array}
      \right\} \\
      \nonumber &=& \sum_{\{S\subseteq [d],\# S\leq j\}}
         (-1)^{j-\#S}{d-\#S \choose j-\#S}
         \# \left\{\sigma \in S_{d+1}^l\,\,\,\,\ \Big| \,\,\,\,\,\,\
         \begin{array}{ll}
           \D^l(\sigma)\subseteq S,  \\
           \sigma(d+1)=d+1-\mu
         \end{array}
      \right\}  \\
      \nonumber &=& \sum_{\{S\subseteq [d],\# S\leq j\}}
        (-1)^{j-\#S}{d-\#S \choose j-\#S}\sum_{T\subseteq S}
          \# \left\{\sigma \in S_{d+1}^l\,\,\,\,\ \Big| \,\,\,\,\,\,\
        \begin{array}{ll}
          \D^l(\sigma)= T,  \\
          \sigma(d+1)=d+1-\mu
        \end{array}
      \right\} \\
      \nonumber &=& \sum_{\{T\subseteq [d],\# T\leq j\}}\# \left\{\sigma \in S_{d+1}^l\,\,\,\,\ \Big| \,\,\,\,\,\,\
       \begin{array}{ll}
         \D^l(\sigma)= T,  \\
         \sigma(d+1)=d+1-\mu
       \end{array}
      \right\} \sum_{\{S\supseteq T, \#S\leq j\}}(-1)^{j-\#S}{d-\#S \choose j-\#S}\\
      \nonumber &=& \sum_{\{T\subseteq [d],\# T\leq j\}}\# \left\{\sigma \in S_{d+1}^l\,\,\,\,\ \Big| \,\,\,\,\,\,\
       \begin{array}{ll}
         \D^l(\sigma)= T,  \\
         \sigma(d+1)=d+1-\mu
       \end{array}
      \right\}\sum_{i=\# T}^j(-1)^{j-i}{d-i \choose j-i}{d-\#T \choose i-\#T}.
      \end{eqnarray}
      But
      \begin{eqnarray}
      \nonumber \sum_{i=\# T}^j(-1)^{j-i}{d-i \choose j-i}{d-\#T \choose i-\#T} &=& {d-\#T \choose i-\#T}\sum_{i=\# T}^j(-1)^{j-i}{j-\#T \choose i-\#T} \\
      \nonumber   &=& \delta_{j,\#T}.
      \end{eqnarray}
      Hence\\
      $$\sum_{k=0}^d\sum_{i=0}^j (-1)^{j-i}{d-i \choose j-i}{d-\mu \choose d-k} R(k,i,l) $$
      \begin{eqnarray}
      \nonumber  &=& \sum_{\{T\subseteq [d],\# T= j\}}\# \left\{\sigma \in S_{d+1}^l\,\,\,\,\ \Big| \,\,\,\,\,\,\
      \begin{array}{ll}
         \D^l(\sigma)= T,  \\
         \sigma(d+1)=d+1-\mu
       \end{array}
      \right\} \\
      \nonumber &=& \# \left\{\sigma \in S_{d+1}^l\,\,\,\,\ \Big| \,\,\,\,\,\,\
       \begin{array}{ll}
         \des^l(\sigma)= j,  \\
         \sigma(d+1)=d+1-\mu
       \end{array}
      \right\}.
      \end{eqnarray}
      This completes the proof.
    \end{proof}

    We note that since for a $(d-1)$-dimensional simplicial complex $\Delta$ and $l \leq d-1$ the
    subdivision operation $\sd^l(\bullet)$ is non-trivial in top-dimension it follows from Theorem
    5.5. from \cite{D} that iterated application of $\sd^l(\bullet)$ will lead to a convergence
    phenomenon for the $f$-vector. More precisely, for a $(d-1)$-dimensional simplicial complex
    $\Delta$, set $\Delta^{(n)} := \underbrace{\sd^l(\cdots \sd^l}_n (\Delta)$ and
    $f^{(n)} (t) = \sum_{i=0}^d f_{i-1}^{\Delta^{(n)}}t^{d-i}$ then for $n \rightarrow \infty$
    one root of $f^{(n)}(t)$ will go to $-\infty$ and the others converge to complex numbers independent
    of $\Delta$. This phenomenon was fist observed in \cite[Thm. 4.2]{BW} for the special case of classical
    barycentric subdivision $\sd^1(\bullet)$. In addition, in \cite[Thm. 3.1]{BW} it is shown that for $l=1$
    the polynomial $f^{(1)}(t)$ has only real roots. Simple examples show that this is not the case
    for general $l$.

\section{The Transformation Matrices}
  \noindent For a $(d-1)$-dimensional simplicial complex $\Delta$ we
  denote by $\mathfrak{H}_{d-1}  = {(h_{ij}^{(d-1)})}_{0\leq i,j \leq d}
  \in \RR^{(d+1) \times (d+1)}$ the matrix of the
  linear transformation that sends the $h$-vector of $\Delta$ to the $h$-vector
  of $\sd(\Delta)$ and $\mathfrak{H}^l_{d-1} = {(h_{ij}^{(d-1,l)})}_{0\leq i,j \leq d}
  \in \RR^{(d+1) \times (d+1)}$
  the matrix of the transformation of the $h$-vector of $\Delta$ to the
  $h$-vector of $\sd^l(\Delta)$. Thus $\mathfrak{H}^1_{d-1} = \mathfrak{H}_{d-1}$.
  By \cite[Thm. 1]{BW} we know $h_{ij}^{(d-1)}=A(d+1,i,j+1)$
  and more generally by Theorem \ref{thm:main} we know $h_{ij}^{(d-1,l)} = A(d+1,i,d+1-j,l)$.

  As an illustration we present the matrices $\mathfrak{H}_d^l$ for
  $d=4$ and $l=3$ and $l=2$.

  \smallskip

  \begin{tabular}{cc}
      \begin{minipage}{0.5\textwidth}
        $\mathfrak{H}_3^3=\left(
          \begin{array}{ccccc}
            1 & 0 & 0 & 0 & 0 \\
            1 & 2 & 1 & 1 & 1 \\
            1 & 1 & 2 & 1 & 1 \\
            1 & 1 & 1 & 2 & 1 \\
            0 & 0 & 0 & 0 & 1 \\
          \end{array}
          \right)
        $
      \end{minipage}
     &
      \begin{minipage}{0.5\textwidth}
        $\mathfrak{H}_3^2=\left(
          \begin{array}{ccccc}
            1 & 0 & 0 & 0 & 0 \\
            5 & 5 & 5 & 2 & 1 \\
            5 & 5 & 6 & 5 & 5 \\
            1 & 2 & 3 & 5 & 5 \\
            0 & 0 & 0 & 0 & 1 \\
          \end{array}
        \right)
       $
      \end{minipage}
    \end{tabular}

  \smallskip

  The following lemma follows immediately from the definition of $A(d+1,i,j+1)$.

  \begin{Lemma}
    \label{sumH}
    The sum of all entries of $\mathfrak{H}^l_{d-1}$ is given by:\\
    $$\sum_{0\leq i,j \leq d}h_{ij}^{(d-1,l)}=\frac{(d+1)!}{l!},$$
    and the sum of all entries of each column is given by:\\
    $$\sum_{0\leq j \leq d}h_{ij}^{(d-1,l)}=\frac{d!}{l!}, \quad 0 \leq i \leq d.$$
  \end{Lemma}
%

  The next simple lemma gives an explicit formula for $\mathfrak{H}^{d-1}_{d-1}$ which
  will serve as the induction base for the proof of monotonocity of the $h$-vector
  under partial barycentric subdivision in Corollary \ref{Co:proof}.

  \begin{Lemma}
    \label{le:entries}
    Let $l=d-1$, then the entries of $\mathfrak{H}^{d-1}_{d-1}$ are given by:\\
    $$h_{ij}^{(d-1,d-1)}=\left\{
                          \begin{array}{ll}
                            0, & \hbox{$i=0,j\not=0$ or $i=d,j\not=d$;} \\
                            2, & \hbox{$i=j=1,\ldots,d-1$;} \\
                            1, & \hbox{otherwise.}
                          \end{array}
                        \right.$$ \\

    and hence
    $$\mathfrak{H}^{d-1}_{d-1}=\left(
                                  \begin{array}{cccccc}
                                    1 & 0 & 0 & \cdots & 0 & 0 \\
                                    1 & 2 & 1 & \cdots & 1 & 1\\
                                    1 & 1 & 2 & \cdots & 1 & 1\\
                                    \vdots & \vdots & \vdots &   & \vdots & \vdots \\
                                    1 & 1 & 1 & \cdots & 2 & 1\\
                                    0 & 0 & 0 & \cdots & 0 & 1 \\
                                  \end{array}
                                \right)
      $$
    \end{Lemma}
    \begin{proof}
      We will prove it by describing the entries of an arbitrary column. Let $C_j=(A(d+1,i,d+1-j,d-1))_{0\leq i\leq d} \in \RR^{(d+1) \times 1}$ be the
      $j^{\text{th}}$ column of $\mathfrak{H}^{d-1}_{d-1}$. Then by definition the entries of $C_j$ count the permutations $\sigma\in S_{d+1}^{d-1}$
      such that $\sigma(d+1)=d+1-j$ according to their number of $l$-descents. By fixing the last element, we are left with $d$ permutations
      $\{\sigma_0,\sigma_1,\ldots,\sigma_{d-1}\}\subset S_{d+1}^{d-1}$, and we arrange them in such a way that
      $\{\sigma_0(d),\sigma_1(d),\ldots,\sigma_{d-1}(d)\}$ are in descending order. We count the
      number of $l$-descents in the following way:

      The number of $l$-descents of $\sigma_i$ in the first $d-1$ positions is $i$ by the way we have arranged $\sigma_0, \ldots, \sigma_d$.
      Moreover since $\sigma_i(d+1)=d+1-j$ there is a descent in position $d$ if and only if $i < j$.
      Therefore, the number of $l$-descents is $i+1$ if $i< j$ and $i$ if $i\geq j$. Now a simple count implies the assertion.
    \end{proof}

    The examples above and the preceding lemma suggest some relations among the entries of $\mathfrak{H}^{l}_{d-1}$ that we verify in the next lemmas.

    \begin{Lemma}
      For $0\leq i,j\leq d$, $$A(d+1,i,d+1-j,l)=A(d+1,d-i,j+1,l).$$
    \end{Lemma}
    \begin{proof}
      Let us denote by $S_{d+1}^l(i,d+1-j)$ the set of permutations $\sigma \in S_{d+1}^l$
      such that $\des^l(\sigma)=i$ and $\sigma(d+1)=d+1-j$. Thus $A(d+1,i,d+1-j,l)=\# S_{d+1}^l(i,d+1-j)$.
      To complete the proof it is enough to provide a bijection between
      $S_{d+1}^l(i,d+1-j)$ and $S_{d+1}^l(d-i,j+1)$.
      Let $$\varphi : S_{d+1}^l(i,d+1-j) \rightarrow S_{d+1}^l(d-i,j+1)$$
      be the map that sends
      $\sigma =(\sigma(1),\ldots,\sigma(d+1)) \in S_{d+1}^l(i,d+1-j)$ to
      $$\varphi(\sigma):=(d+2-\sigma(l),\ldots,d+2-\sigma(1),d+2-\sigma(l+1),
          \ldots,d+2-\sigma(d+1)).$$

      Since $\sigma(1),\ldots,\sigma(l)$ are in descending order, we have
      that $d+2-\sigma(l),\ldots,d+2-\sigma(1)$ are also in descending order
      and hence $\varphi(\sigma)\in S_{d+1}^l$.
      By definition $\phi(\sigma)(d+1) = d+1 - \sigma(d+1) = d+1-j$.
      Thus to show that $\varphi(\sigma) \in S_{d+1}^l(d-i,j+1)$
      it remains to verify that the number of $l$-descents of
      $\varphi(\sigma)$ is $d-i$.

      We show that $m \in [d]$ is an $l$-descent of $\sigma$ if and only if
      $m$ is not an $l$-descent of $\varphi(\sigma)$.
      If $m \in[l]$ then
      $\sigma(j)>\sigma(l+1)$ implies
      $d+2-\sigma(j)<d+2-\sigma(l+1)$ and
      $\sigma(j)<\sigma(l+1)$ implies
      $d+2-\sigma(j)>d+2-\sigma(l+1)$.
      Analogously, if $m\in [d]\setminus [l]$
      then $\sigma(m)>\sigma(m+1)$ implies $d+2-\sigma(m)<d+2-\sigma(m+1)$
      and
      $\sigma(m)<\sigma(m+1)$ implies $d+2-\sigma(m)>d+2-\sigma(m+1)$.

      Therefore, the number of $l$-descents of $\varphi(\sigma)$ is $d-i$.
      This completes the proof since $\varphi$ is clearly a bijection.
    \end{proof}

    \begin{Proposition}
      \label{pr:ineq}
      For $0\leq i,j \leq d$,
      \begin{eqnarray}
        \label{eq:ineq}
        A(d+1,i,d+1-j,l+1) & \leq & A(d+1,i,d+1-j,l).
      \end{eqnarray}
      In addition, for $d\geq 4,$ $0\leq j\leq d$ and $2\leq i \leq d-2$, inequality
      \eqref{eq:ineq} is strict.
    \end{Proposition}
    \begin{proof}
      For the sake of short notation, within the proof we say descent
      for corresponding $l$-descent and $``\,\,\hat{ }\,\,"$ means that the entry is
      missing in the permutation.
      We define a map $$\psi:S_{d+1}^{l+1}(i,d+1-j)\rightarrow S_{d+1}^l(i,d+1-j)$$
      as follows:

      Let $\sigma\in S_{d+1}^{l+1}(i,d+1-j)$ be a permutation for which $p$ is the
      number of descents in the first $l+1$ positions and $i-p$ descents in the
      remaining positions for some $0 \leq p \leq l$.
      Thus we can write
      $\sigma=(\sigma(1),\ldots,\sigma(p),\ldots,
           \sigma(l+1),\sigma(l+2),\ldots,\sigma(d+1))$
      such that $\sigma(p)> \sigma(l+2)$ and $\sigma(p+1)<\sigma(l+2)$,
      where $\sigma(1),\ldots,\sigma(l+1)$ are in descending order and
      $\sigma(d+1)=d+1-j$.\\

      We define
      $$\psi(\sigma)=(\sigma(1),\ldots,\hat{\sigma}(p+1),\ldots,\sigma(l+1),
      \sigma(p+1),\ldots,\sigma(d));$$ i.e we change the position of
      $\sigma(p+1)$ from $p+1$ to $l+1$. It is easy to see that $\psi(\sigma)$ is a
      permutation for which $p$ is the number of descents in the
      first $l$ position and $i-p$ descents in the remaining position with
      $\psi(\sigma)(d)=d+1-j$. Therefore $\psi(\sigma)\in S_{d+1}^l(i,d+1-j)$.
      Clearly, $\psi$ is injective and hence we have
      $S_{d+1}^{l+1}(i,d+1-j)\subseteq S_{d+1}^l(i,d+1-j)$ which implies
      $A(d+1,i,d+1-j,l+1) \leq A(d+1,i,d+1-j,l)$.

      Now assume $d\geq 4,$ $0\leq j\leq d$ and $2\leq i \leq d-2$.
      For the proof of the strict inequality in \eqref{eq:ineq} it suffices
      to find at least one element
      $\sigma\in S_{d+1}^l(i,d+1-j)$ that does not have a preimage under $\psi$
      in $S_{d+1}^{l+1}(i,d+1-j)$. We consider two cases:

      \smallskip

      \noindent\textsf{Case 1 $(j\not=0)$}: We set $\sigma(d+1)=d+1-j$ and $\sigma(d)=d+1$.
      Now we are left with $d-1$ elements to be arranged with $i-1$ descents.
      Let $\rho_1,\ldots,\rho_{d-1}$ be the remaining elements of $[d+1]$ arranged in
      ascending order; i.e. $\rho_s < \rho_t$ for $s<t$. We have further two
      possibilities:\\

      \begin{itemize}
        \item[(a)] $i-1\leq l$. Then reordering the elements as
          $$\rho_{l+1},\ldots,\hat{\rho}_{l+2-i},\ldots,\rho_1,\rho_{l+2-i},\rho_{l+2}, \ldots,\rho_{d-1}$$ yields the required number of
          $l$-descents. Clearly, the formal preimage of $\sigma$ under $\psi$
          has not belong to $S_{d+1}^{l+1}(i,d+1-j)$ since it has only one descent but $i\geq 2$.
        \item[(b)] $i-1> l$. Then we reorder the elements in the first $l+1$ positions as
          $$\rho_{l+1},\ldots,\rho_2,\rho_1,$$
          which contributes $l$ to the number of descents. The remaining $d-1-(l+1)$ elements can be
          arranged in such a way that they contribute $i-1-l$ to the number of descents. By this setting we cannot have a descent at
          $(l+1)^{\text{th}}$ and $(d-1)^{\text{st}}$ position, so at most we can have $d-2$ descents which coincides with the upper bound for $i$. Again, the formal
          preimage of $\sigma$ under $\psi$ does not belong to $S_{d+1}^{l+1}(i,d+1-j)$ since its number
          of descents are $i-l$.
      \end{itemize}
      \noindent\textsf{Case 2 $(j=0)$}: We set $\sigma(d+1):=d+1$. Thus we are left with the $d$ elements $[d]$
      to be arranged yielding a permutation in $S_{d+1}^l(i,d+1-l)$. The same arrangement as in Case 1 with
      $d$ elements and $i$ descents will give us the required element $\sigma$.
    \end{proof}

    As a consequence of Theorem \ref{thm:main} and Proposition \ref{pr:ineq} we can deduce a result on the growth of the $h$-vector under
    $l$\textsuperscript{th} partial barycentric subdivision.

    \begin{Corollary}
     \label{Co:proof}
      Let $\Delta$ be a $d$-dimensional simplicial complex such that $h_i^\Delta \geq 0$ for all $0 \leq i \leq d$.
      Then $h_i^\Delta \leq h_i^{\sd^l(\Delta)}$ for $0 \leq i \leq d$.
    \end{Corollary}
    \begin{proof}
      It is easy to see that $h_0^{\sd^l(\Delta)}=h_0^{\Delta}$ and $h_d^{\sd^l(\Delta)}=h_d^{\Delta}$, thus we are left with the case $1\leq i \leq d-1.$
      Since by Theorem \ref{thm:main} $h_i^{\sd^l(\Delta)}$ is a non-negative linear combination of the $h_j^\Delta$ it suffices to show that the entries
      of the submatrix ${(h_{ij}^{(d-1,l)})}_{1\leq i\leq d-1,\,\, 0\leq j \leq d}$ are non-zero. Again, by equation \ref{eq:ineq} it is enough to
      consider the case $l=d-1$. By Lemma \ref{le:entries} we complete the proof.
    \end{proof}

    The consequence of the preceding corollary for the smaller class of Cohen-Macaulay simplicial complexes
    also follows from a very general result by Stanley \cite[Theorem 4.10]{St} using the fact that $l$\textsuperscript{th}
    partial barycentric subdivision is a quasi-geometric subdivision. Note that $h_i^\Delta \geq 0$ for
    Cohen-Macaulay simplicial complexes.

    Let $\mathfrak{F}_{d-1}$ be the matrix of the transformation that sends $f$-vector of $\Delta$ to
    $f$-vector of $\sd(\Delta)$. We denote by $\mathfrak{F}_{d-1}^l$ the matrix of the transformation from
    the $f$-vector of $\Delta$ to the $f$-vector of $\sd^l(\Delta)$. Both matrices $\mathfrak{F}_{d-1}$ and
    $\mathfrak{F}^l_{d-1}$ are square matrices of order $d+1$, with $\mathfrak{F}_{d-1}^1=\mathfrak{F}_{d-1}$.
    By Theorem \ref{le:fvector} the entries of $\mathfrak{F}^l_{d-1}={(f_{ij}^{(d-1,l)})}_{0\leq i,j \leq d}$ are given by
    $f_{ij}^{(d-1,l)}=R(j,i,l)$.

    The following lemma shows that the matrices $\mathfrak{F}^l_{d-1}$ and $\mathfrak{H}^l_{d-1}$ are diagonalizable.

    \begin{Proposition}
    \label{le:diagonalizble}
      For $1\leq l\leq d-1$:
      \begin{enumerate}
        \item The matrices $\mathfrak{F}^l_{d-1}$ and $\mathfrak{H}^l_{d-1}$ are
          similar.
        \item The matrices $\mathfrak{F}^l_{d-1}$ and $\mathfrak{H}^l_{d-1}$ are
          diagonalizable with eigenvalue $1$ of multiplicity $l+1$ and
          eigenvalues $\frac{(l+1)!}{l!},\ldots,\frac{d!}{l!}$ of multiplicity $1$.
      \end{enumerate}
    \end{Proposition}
    \begin{proof}
      \begin{itemize}
        \item Since the transformation sending the $f$-vector of a simplicial complex to the
           $h$-vector of a simplicial complex is an invertible linear transformation, the first
           assertion follows.
        \item Clearly, $\mathfrak{F}^l_{d-1}$ is an upper triangular matrix with
           diagonal entries
           $$\underbrace{1,\ldots,1}_{(l+1)-\mbox{times}},
               \frac{(l+1)!}{l!},\ldots,\frac{d!}{l!}.$$
           Let $(\mathfrak{F}^l_{d-1})^\perp$ be the transpose of $\mathfrak{F}^l_{d-1}$.
           Then for $(\mathfrak{F}^l_{d-1})^\perp$, the first $(l+1)$ unit vectors are
           eigenvectors for the eigenvalue $1$. Also, The eigenvalues
           $\frac{(l+1)!}{l!},\ldots,\frac{d!}{l!}$ are pairwise different. This implies
           that $(\mathfrak{F}^l_{d-1})^\perp$ is diagonalizable.
           But then $\mathfrak{F}^l_{d-1}$ is diagonalizable.
      \end{itemize}
    \end{proof}

    \begin{Corollary}
    \label{Co:initialeigenvectors}
      Let $\nu=(\nu_0,\ldots,\nu_{d})$ be an eigenvector of the matrix
      $\mathfrak{H}^l_{d-1}$ for the eigenvalue $\lambda$ such that
      $\lambda\not=\frac{d!}{l!}$. Then $\sum_{i=0}^d \nu_{i}=0.$
    \end{Corollary}
    \begin{proof}
      Since
      \begin{eqnarray}
        \nonumber  \mathfrak{H}^l_{d-1}\nu  &=& \lambda \nu \\
        \Rightarrow \nonumber   (1,\ldots,1)\mathfrak{H}^l_{d-1}\nu &=& (1,\ldots,1)\lambda \nu
      \end{eqnarray}
      But by Lemma \ref{sumH}, $(1,\ldots,1)\mathfrak{H}^l_{d-1}=\frac{d!}{l!}(1,\ldots,1)$.
      Therefore, either $\lambda=\frac{d!}{l!}$ or
      $\sum_{i=0}^d \nu_{i}=0$. Since $\lambda\not=\frac{d!}{l!}$ we are done.
    \end{proof}

    Next we try to gain a better understanding of the eigenvectors of $\mathfrak{H}^l_{d-1}$.

    \begin{Lemma}
    \label{le:Fd}
      Let $d\geq 2$ and $\nu_1^{(1)},\ldots,\nu_1^{(l+1)},\nu_{l+1},\ldots,\nu_d$ be a basis of
      eigenvectors of the matrix $\mathfrak{F}^l_{d-1}$, where $\nu_1^{(1)},\ldots,\nu_1^{(l+1)}$ are
      eigenvectors for the eigenvalue $1$ and $\nu_{l+1},\ldots,\nu_d$ are eigenvectors for the
      eigenvalues $\{l+1,\ldots,d\}$, respectively. Then
      $(\nu_1^{(1)},0),\ldots,(\nu_1^{(l+1)},0),(\nu_{l+1},0),\ldots,(\nu_d,0)$ are eigenvectors of
      the matrix $\mathfrak{F}^l_d$ for the eigenvalues $\{\underbrace{1,\ldots,1}_{(l+1) times},l+1,\ldots,d\}.$
    \end{Lemma}
    \begin{proof}
      Since both $\mathfrak{F}^l_{d-1}$ and $\mathfrak{F}^l_{d}$ are upper triangular and $\mathfrak{F}^l_{d-1}$ is
      obtained by deleting $(d+2)$\textsuperscript{nd} column and row from $\mathfrak{F}^l_{d}$ the assertion follows.
    \end{proof}

    Let $\hat{\mathfrak{H}}^l_{d-1}$ be a matrix obtained be deleting the first and last rows and columns of
    $\mathfrak{H}^l_{d-1}$. Thus $\hat{\mathfrak{H}}^l_{d-1}$ is a $d$ by $d$ square matrix.

    \begin{Lemma}
    \label{le:innermatrixH}
      The matrix $\hat{\mathfrak{H}}^l_{d-1}$ is diagonalizable.
    \end{Lemma}
    \begin{proof}
      By definition and \ref{thm:main} the first row of $\mathfrak{H}^l_{d-1}$ is the first unit vector and the
      last row of $\mathfrak{H}^l_{d-1}$ is the $(d+1)$\textsuperscript{st} unit vector. Thus the characteristic
      polynomial of $\mathfrak{H}^l_{d-1}$ splits into $(1-t)^2$ times the characteristic polynomial of
      $\hat{\mathfrak{H}}^l_{d-1}$. Therefore, $\hat{\mathfrak{H}}^l_{d-1}$ has eigenvalues
      $$\underbrace{1,\ldots,1}_{(l-1)-\mbox{times}},\frac{(l+1)!}{l!},\ldots,\frac{d!}{l!}.$$
      To show the matrix $\hat{\mathfrak{H}}^l_{d-1}$ is diagonalizable, it is enough to show that the eigenspace for
      the eigenvalue $1$ is of dimension $l-1$.

      For this we again consider the full matrix $\mathfrak{H}^l_{d-1}$. Since $\mathfrak{H}^l_{d-1}$ is diagonalizable
      there is a basis $\omega_1^{(1)},\ldots,$ $\omega_1^{(l+1)},$ $\omega_{l+1},\ldots,\omega_{d}$ of $\RR^{d+1}$
      consisting of eigenvectors of $\mathfrak{H}^l_{d-1}$. We can choose the numbering such that
      $\omega_1^{(i)},1\leq i\leq l+1$ are eigenvectors for the eigenvalue $1$
      and $\omega_j$ is an eigenvector for the eigenvalues $\frac{j!}{l!},l+1 \leq j\leq d$, respectively.

      Again, since the first and last row of $\mathfrak{H}^l_{d-1}$ are the first and $(d+1)$\textsuperscript{st}
      unit vector we can choose the eigenvectors of $\mathfrak{H}^l_{d-1}$ for the eigenvalue $\lambda=1$ as follows:
      $\omega_1^{(1)}$ and $\omega_1^{(2)}$ can be chosen such that
      $$\omega_1^{(1)}=(1,k_{11},\ldots,k_{1(d-1)},0) \text{ and } \omega_1^{(2)}=(0,k_{21},\ldots,k_{2(d-1)},1),$$
      and $\omega_1^{(i)}$ can be chosen such that $\omega_1^{(i)}=(0,k_{i1},\ldots,k_{i(d-1)},0)$ for
      $3\leq i\leq l+1$. Clearly, this implies that
      deleting the leading and trailing $0$ form the $\omega_1^{(i)}$ for $3\leq i\leq l+1$ yields
      eigenvectors $\hat{\omega_1}^{(i)}=(k_{i1},\ldots,k_{i(d-1)})$ of $\hat{\mathfrak{H}}^l_{d-1}$
      for the eigenvalue $\lambda=1$. Obviously, the set of vectors
      $\{\hat{\omega_1}^{(3)},\ldots,\hat{\omega_1}^{(l+1)}\}$ is linearly independent. Hence we have
      shown that the dimension of the eigenspace for the eigenvalue $1$ of $\hat{\mathfrak{H}}^l_{d-1}$
      $l-1$.
    \end{proof}

    The above lemma is a key ingredient in proving the following theorem.

    \begin{Theorem}
    \label{th:eigenvector}
      Let $d\geq 2$ and let $\omega_1^{(1)},\ldots,\omega_1^{(l+1)},\omega_{l+1},\ldots,\omega_{d}$ be a basis of
      eigenvectors of the matrix $\mathfrak{H}^l_{d-1}$, where $\omega_1^{(i)},1\leq i\leq l+1$ are eigenvectors
      for the eigenvalue $1$ and $\omega_j$ is an eigenvector for the eigenvalue $\frac{j!}{l!},l+1 \leq j\leq d$.
      \begin{enumerate}
        \item Let $\Delta$ be a $(d-1)$-dimensional simplicial complex. If we expand $h$-vector of $\Delta$ in
          terms of eigenvectors of the matrix $\mathfrak{H}^l_{d-1}$, the coefficient of the eigenvector for the
          eigenvalue $\frac{d!}{l!}$ is non-zero.
        \item The first and the last coordinate entry in
          $\omega_1^{(3)},\ldots,\omega_1^{(l+1)},\omega_{l+1},\ldots,\omega_{d}$ is zero.
        \item The vectors $\omega_1^{(1)}$ and $\omega_1^{(2)}$ can be chosen such that
          $$\omega_1^{(1)}=(1,i_1,\ldots,i_{d-1},0)\text{ and }\omega_1^{(2)}=(0,j_1,\ldots,j_{d-1},1).$$
        \item The vector $\omega_d$ can be chosen such that
          $\omega_d=(0,b_1,\ldots,b_{d-1},0)$ for strictly positive rational numbers $b_i$, $1\leq i\leq d-1$.
      \end{enumerate}
    \end{Theorem}
    \begin{proof}
      Let us expand the $f$-vector of $\Delta$ in terms of a basis of eigenvectors of the matrix
      $\mathfrak{F}^l_{d-1}$. Since $f^{\Delta}_{d-1}\not=0$ from Lemma \ref{le:Fd} we deduce that the
      coefficient of the eigenvector for the highest eigenvalue is non-zero.
      Since $\mathfrak{F}^l_{d-1}$ and $\mathfrak{H}^l_{d-1}$ are similar so (1) follows.

      Assertions (2) and (3) immediately follow from the proof of Lemma \ref{le:innermatrixH}.

      For (4) consider the matrix $\hat{\mathfrak{H}}^l_{d-1}$ as defined above. It is easily seen that the
      entries of $\hat{\mathfrak{H}}^l_{d-1}$ are strictly positive numbers. Therefore, by the
      Perron-Frobenius Theorem \cite{HJ} it follows that there is an eigenvector $\hat{\omega}^l_d$
      for the eigenvalue $\frac{d!}{l!}$ with strictly positive entries.
      Hence $(0,\hat{\omega}^l_d,0)$ is the required eigenvector.
    \end{proof}

\section{Open Problems}
    In this section we discuss some of the open problems related to the above work.

    Corollary \ref{Co:initialeigenvectors} describes properties of the eigenvectors of the matrix
    $\mathfrak{H}^l_{d-1}$ for the eigenvalue $\lambda$ such that $\lambda\not=\frac{d!}{l!}$.
    For the eigenvalue $\lambda=\frac{d!}{l!}$ we were able to deduce its non-negativity in
    Theorem \ref{th:eigenvector} (4) but were not able to give more structural results or even
    provide an explicit description.
    By \cite{D} when applying $l$\textsuperscript{th} partial barycentric subdivision iteratively
    the limiting behavior of the $h$-vector is determined by this eigenvector, in a sense
    specified in \cite{D}. Hence some information can be read off from \cite{D} nevertheless
    complete information about that eigenvector would be desirable.

    For example, for $d=4$ we have following eigenvectors, corresponding to the
    eigenvalues $\frac{4!}{3!}, \frac{4!}{2!}, \frac{4!}{1!}$, respectively.
    $$\left(
        \begin{array}{c}
          0 \\
          1 \\
          1 \\
          1 \\
          0 \\
        \end{array}
      \right),
      \left(
        \begin{array}{c}
          0 \\
          1 \\
          \frac{5}{3} \\
          1 \\
          0 \\
        \end{array}
      \right),
      \left(
        \begin{array}{c}
          0 \\
          1 \\
          \frac{7}{2} \\
          1 \\
          0 \\
        \end{array}
      \right).
    $$
    For $d=5$ we have the following eigenvectors, corresponding to the eigenvalues
    $\frac{5!}{4!}, \frac{5!}{3!}, \frac{5!}{2!}, \frac{5!}{1!}$, respectively.
    $$
    \left(
        \begin{array}{c}
          0 \\
          1 \\
          1 \\
          1 \\
          1 \\
          0 \\
        \end{array}
      \right),
      \left(
        \begin{array}{c}
          0 \\
          1 \\
          \frac{12}{7} \\
          \frac{12}{7} \\
          1 \\
          0 \\
        \end{array}
      \right),
      \left(
        \begin{array}{c}
          0 \\
          1 \\
          \frac{46}{11} \\
          \frac{46}{11} \\
          1 \\
          0 \\
        \end{array}
      \right),
      \left(
        \begin{array}{c}
          0 \\
          1 \\
          \frac{17}{2} \\
          \frac{17}{2} \\
          1 \\
          0 \\
        \end{array}
      \right).
    $$
    Similarly, for $d=6$ we have following eigenvectors, corresponding to the eigenvalues
    $\frac{6!}{5!},\frac{6!}{4!}, \frac{6!}{3!}, \frac{6!}{2!}, \frac{6!}{1!}$, respectively.
    $$
    \left(
        \begin{array}{c}
          0 \\
          1 \\
          1 \\
          1 \\
          1 \\
          1 \\
          0 \\
        \end{array}
      \right),
      \left(
        \begin{array}{c}
          0 \\
          1 \\
          \frac{7}{4} \\
          \frac{7}{4} \\
          \frac{7}{4} \\
          1 \\
          0 \\
        \end{array}
      \right),
      \left(
        \begin{array}{c}
          0 \\
          1 \\
          \frac{1941}{437} \\
          \frac{2146}{437} \\
          \frac{1941}{437} \\
          1 \\
          0 \\
        \end{array}
      \right),
      \left(
        \begin{array}{c}
          0 \\
          1 \\
          \frac{5431}{527} \\
          \frac{8906}{527} \\
          \frac{5431}{527} \\
          1 \\
          0 \\
        \end{array}
      \right),
      \left(
        \begin{array}{c}
          0 \\
          1 \\
          \frac{586}{33} \\
          \frac{5459}{132} \\
          \frac{586}{33} \\
          1 \\
          0 \\
        \end{array}
      \right).
    $$
    Thus the following problem appears to be interesting.
    \begin{Problem}
      Give a description of eigenvectors of the matrices $\mathfrak{F}^l_{d-1}$ and $\mathfrak{H}^l_{d-1}$
      for the eigenvalue $\frac{d!}{l!}$.
    \end{Problem}

    Let $V$ be a vertex set such that $\# V=d$ and let $2^V$ denote the simplex with vertex set $V$.
    Let $\Gamma$ be the first barycentric subdivision of $2^V$.
    The $h$-polynomial $h(\Gamma,x) = \sum_{i=0}^{d} h_ix^{d-i}$ of $\Gamma$ has the following combinatorial
    interpretation.
    \begin{equation}
    \label{hvectorexc}
      h(\Gamma,x) = \sum_{\sigma \in S_d}x^{\des(\sigma)} =\sum_{\sigma \in S_d}x^{\ex(\sigma)},
    \end{equation}
    where $\ex(\sigma)$ denotes the number of {\em excedances} of $\sigma$, defined by
    $$\ex(\sigma)=\# \{i ~|~\sigma(i)>i\},$$
    The first equality follows from \cite[Theorem 3.13.1]{St2} (it is also a consequence of \cite[Thm 1]{BW} and
    Theorem \ref{thm:main}),
    and the second is a consequence of \cite[Proposition 1.4.3]{St2}.
    In \cite{St}, the {\em local $h$-polynomial} $\ell_V(\Gamma,x)$ of $\Gamma$ has been defined and given in a
    similar passion of equation
    \eqref{hvectorexc} as follows:
    \begin{equation}
    \label{localhvector}
      \ell_V(\Gamma,x)=\sum_{\sigma\in D_d}x^{\ex(\sigma)},
    \end{equation}
    where $D_d$ denotes the set of all derangements in $S_d$. We suggest the followings:

    \begin{Problem}
      Give an interpretation of local $h$-polynomial for the $l$\textsuperscript{th} partial barycentric
      subdivision similar to \eqref{localhvector} in terms of a suitably defined $l$-excedance statistic
      on a newly defined set of $l$-derangements satisfying an analog of \eqref{hvectorexc}
    \end{Problem}

    Already the question of finding a statistic on $S_d^l$ fulfilling a statement analogous to
    \eqref{hvectorexc} seems to be hard and challenging.
    We note that in \cite{Ath} a theory of local $\gamma$-vectors of subdivisions was initiated.

    \begin{Problem}
    \label{problm 3}
      Define an $l$-excedance statistic on $S_d^l$ such that the $l$-excedance and $l$-descent statistic
      on $S_d^l$ are equally distributed; i.e. satisfy an analog of \eqref{localhvector}.
    \end{Problem}

    For Problem \ref{problm 3} we tried, different approaches. Despite not yielding a solution to the
    problem the following idea resulted in some interesting data.
    We define an injective map say $\chi:S_d^l \rightarrow S_d$ in the following way. Let $\sigma\in S_d^l$
    such that $\sigma=(\sigma(1),\sigma(2),\ldots,\sigma(d)),$ with first $l$ elements are in descending order, then:
    $$\chi(\sigma)=\left\{
                     \begin{array}{ll}
                       (\sigma(l),\ldots,\sigma(1),\sigma(l+1),\ldots,\sigma(d)), & \hbox{if $\sigma(l+1)>\sigma(1)$} \\
                        & \hbox{or $\sigma(l)>\sigma(l+1)$;} \\
                       (\sigma(l_1),\ldots,\sigma(1),\sigma(l),\ldots,\sigma(l_1+1), & \\
                       \sigma(l+1),\ldots, \sigma(d)), &  \hbox{if $\sigma(l_1)>\sigma(l+1)$}\\
                       &  \hbox{and $\sigma(l_1+1)<\sigma(l+1)$}.\\
                     \end{array}
                   \right.$$
    Now define the number of $l$-excedances $\ex^l(\sigma)$ of $\sigma\in S_d^l$ to be number of usual excedances
    of $\chi(\sigma)$, i.e.  $$\ex^l(\sigma):=\#\{i~|~\chi(\omega)(i)>i\}.$$
    We apply this definition for different values of $d$ and $l$. For a fixed $d$, the
    $l$-descent and $l$-excedance statistic are equally distributed on $S_d^l$ for $l=d-1$ and $l=d-2$.
    But for other values of $l$ the two statistics appear to be different. Nevertheless, the obtained data has some
    surprising and unexplained symmetry. For example, for
    $S_5^l$ we have following tables for the number of $l$-descents,
    \begin{center}
      \begin{tabular}{cc|c|c|c|c}
        \cline{3-5}
        & & \multicolumn{3}{|c|}{$l=$} \\ \cline{3-5}
        & & \cellcolor[gray]{.8} 4 & \cellcolor[gray]{.8} 3 & \cellcolor[gray]{.8} 2 \\ \cline{1-5}
        \multicolumn{1}{|c|}{\multirow{2}{*}{$\#$ of $l$-descents =}} &
        \multicolumn{1}{|c|}{\cellcolor[gray]{.8} 0} & 1 & 1 & 1 &      \\ \cline{2-5}
        \multicolumn{1}{|c|}{}                        &
        \multicolumn{1}{|c|}{\cellcolor[gray]{.8} 1} & 1 & 6 & 16 &     \\ \cline{2-5}
        \multicolumn{1}{|c|}{}                        &
        \multicolumn{1}{|c|}{\cellcolor[gray]{.8} 2} & 1 & 6 & 26 &      \\ \cline{2-5}
        \multicolumn{1}{|c|}{}                        &
        \multicolumn{1}{|c|}{\cellcolor[gray]{.8} 3} & 1 & 6 & 16 &     \\ \cline{2-5}
        \multicolumn{1}{|c|}{}                        &
        \multicolumn{1}{|c|}{\cellcolor[gray]{.8} 4} & 1 & 1 & 1 &      \\ \cline{1-5}
      \end{tabular}
    \end{center}
    and the following table for the number of $l$-excedances.
    \begin{center}
      \begin{tabular}{cc|c|c|c|c}
        \cline{3-5}
        & & \multicolumn{3}{|c|}{$l=$} \\ \cline{3-5}
        & & \cellcolor[gray]{.8} 4 & \cellcolor[gray]{.8} 3 & \cellcolor[gray]{.8} 2 \\ \cline{1-5}
        \multicolumn{1}{|c|}{\multirow{2}{*}{$\#$ of $l$-excedances =}} &
        \multicolumn{1}{|c|}{\cellcolor[gray]{.8} 0} & 1 & 1 & 1 &      \\ \cline{2-5}
        \multicolumn{1}{|c|}{}                        &
        \multicolumn{1}{|c|}{\cellcolor[gray]{.8} 1} & 1 & 6 & 14 &     \\ \cline{2-5}
        \multicolumn{1}{|c|}{}                        &
        \multicolumn{1}{|c|}{\cellcolor[gray]{.8} 2} & 1 & 6 & 30 &      \\ \cline{2-5}
        \multicolumn{1}{|c|}{}                        &
        \multicolumn{1}{|c|}{\cellcolor[gray]{.8} 3} & 1 & 6 & 14 &     \\ \cline{2-5}
        \multicolumn{1}{|c|}{}                        &
        \multicolumn{1}{|c|}{\cellcolor[gray]{.8} 4} & 1 & 1 & 1 &      \\ \cline{1-5}
      \end{tabular}
    \end{center}
    Similarly, for $S_6^l$ the number of $l$-descents are shown in the following table,
    \begin{center}
      \begin{tabular}{cc|c|c|c|c|c}
        \cline{3-6}
        & & \multicolumn{4}{|c|}{$l=$} \\ \cline{3-6}
        & & \cellcolor[gray]{.8} 5 & \cellcolor[gray]{.8} 4 & \cellcolor[gray]{.8} 3 & \cellcolor[gray]{.8} 2 \\ \cline{1-6}
        \multicolumn{1}{|c|}{\multirow{2}{*}{$\#$ of $l$-descents =}} &
        \multicolumn{1}{|c|}{\cellcolor[gray]{.8} 0} & 1 & 1 & 1 & 1 &     \\ \cline{2-6}
        \multicolumn{1}{|c|}{}                        &
        \multicolumn{1}{|c|}{\cellcolor[gray]{.8} 1} & 1 & 7 & 22 & 42 &     \\ \cline{2-6}
        \multicolumn{1}{|c|}{}                        &
        \multicolumn{1}{|c|}{\cellcolor[gray]{.8} 2} & 1 & 7 & 37 & 137 &     \\ \cline{2-6}
        \multicolumn{1}{|c|}{}                        &
        \multicolumn{1}{|c|}{\cellcolor[gray]{.8} 3} & 1 & 7 & 37 & 137 &     \\ \cline{2-6}
        \multicolumn{1}{|c|}{}                        &
        \multicolumn{1}{|c|}{\cellcolor[gray]{.8} 4} & 1 & 7 & 22 & 42 &     \\ \cline{2-6}
        \multicolumn{1}{|c|}{}                        &
        \multicolumn{1}{|c|}{\cellcolor[gray]{.8} 5} & 1 & 1 & 1 & 1 &     \\ \cline{1-6}
      \end{tabular}
    \end{center}
    and the number of $l$-excedances are shown in the following table.
    \begin{center}
      \begin{tabular}{cc|c|c|c|c|c}
        \cline{3-6}
        & & \multicolumn{4}{|c|}{$l=$} \\ \cline{3-6}
        & & \cellcolor[gray]{.8} 5 & \cellcolor[gray]{.8} 4 & \cellcolor[gray]{.8} 3 & \cellcolor[gray]{.8} 2 \\ \cline{1-6}
        \multicolumn{1}{|c|}{\multirow{2}{*}{$\#$ of $l$-excedances =}} &
        \multicolumn{1}{|c|}{\cellcolor[gray]{.8} 0} & 1 & 1 & 1 & 1 &     \\ \cline{2-6}
        \multicolumn{1}{|c|}{}                        &
        \multicolumn{1}{|c|}{\cellcolor[gray]{.8} 1} & 1 & 7 & 17 & 33 &     \\ \cline{2-6}
        \multicolumn{1}{|c|}{}                        &
        \multicolumn{1}{|c|}{\cellcolor[gray]{.8} 2} & 1 & 7 & 42 & 146 &     \\ \cline{2-6}
        \multicolumn{1}{|c|}{}                        &
        \multicolumn{1}{|c|}{\cellcolor[gray]{.8} 3} & 1 & 7 & 42 & 146 &     \\ \cline{2-6}
        \multicolumn{1}{|c|}{}                        &
        \multicolumn{1}{|c|}{\cellcolor[gray]{.8} 4} & 1 & 7 & 17 & 33 &     \\ \cline{2-6}
        \multicolumn{1}{|c|}{}                        &
        \multicolumn{1}{|c|}{\cellcolor[gray]{.8} 5} & 1 & 1 & 1 & 1 &     \\ \cline{1-6}
      \end{tabular}
    \end{center}

    We close by briefly mentioning an interesting problem relating to the $\gamma$-vector
    introduced by Gal \cite{G}. The behavior of the $\gamma$-vector under barycentric
    subdivision was studied in \cite{N}. In \cite{Ath} a theory of local $\gamma$-vectors are started and
    in \cite[Theorem 1.5]{AthSav} a nice interpretation of the local $\gamma$-vector was given for
    the classical barycentric subdivision. Again one can ask the same questions for the $\gamma$ and
    local $\gamma$-vector of the $l$\textsuperscript{th} partial barycentric subdivision.

\end{document}